\documentclass[]{amsart}
\usepackage{mathrsfs}
\usepackage{graphicx}

\usepackage{amssymb}
\usepackage{amsmath}
\usepackage{amsthm}
\usepackage{galois}
\usepackage[all,cmtip]{xy}

\allowdisplaybreaks                 % 多行公式自动分页
\usepackage{float}                 % 控制图片位置宏包
\restylefloat{figure}              % 让不浮动(\begin{figure}[H])能用
\usepackage{caption2}              % 定制浮动图形的标题之标记
\usepackage{array}              % 使表格中\setlength{\doublerulesep}{0pt}不失灵
\usepackage{slashbox}              % 使表格中能画斜线
\newtheorem{thm}{Theorem}[section]
\newtheorem{cor}[thm]{Corollary}
\newtheorem{lem}[thm]{Lemma}
\newtheorem{prop}[thm]{Proposition}
\newtheorem{claim}[thm]{Claim}
\theoremstyle{definition}
\newtheorem{defn}[thm]{Definition}
\newtheorem{exam}[thm]{Example}
\newtheorem{ques}[thm]{Question}
\theoremstyle{remark}

\numberwithin{equation}{section}
\newcommand{\Z}{\mathbb Z}
\newcommand{\Q}{\mathbb Q}

\newcommand{\fix}{\mathrm{Fix}}
\newcommand{\ind}{\mathrm{ind}}

\newcommand{\rk}{\mathrm{rank}}

\newcommand{\edo}{\mathrm{End}}

\newcommand{\ir}{\mathrm{Ir}}
\newcommand{\eq}{\mathrm{Eq}}
\newcommand{\F}{\mathbf{F}}      %fixed point class F
\newcommand{\intt}{\mathrm{int}}

\begin{document}

\title{The group fixed by a family of endomorphisms of a surface group}

\author{Jianchun Wu}
\address{Department of Mathematics, Soochow University, Suzhou 215006, CHINA}
\email{wujianchun@suda.edu.cn}

\author{Qiang Zhang}
\address{School of Mathematics and Statistics, Xi'an Jiaotong University, Xi'an 710049, CHINA}
\email{zhangq.math@mail.xjtu.edu.cn}

%\thanks{Partially supported by the Fundamental Research Funds for the Central Universities.}
%The author thanks Professor Boju Jiang and Shicheng Wang for valuable communications and suggestions.}



%\date{}%
%\dedicatory{}%
%\commby{}%
% ----------------------------------------------------------------
\begin{abstract}
For a closed surface $S$ with $\chi(S)<0$, we show that the fixed subgroup of a family $\mathcal B$ of endomorphisms of $\pi_1(S)$ has $\rk \fix\mathcal B\leq \rk \pi_1(S)$. In particular, if $\mathcal B$ contains a non-epimorphic endomorphism, then $\rk \fix\mathcal B\leq \frac{1}{2} \rk \pi_1(S)$. We also show that geometric subgroups of $\pi_1(S)$ are inert, and hence the fixed subgroup of a family of epimorphisms of $\pi_1(S)$ is also inert.
\end{abstract}
\maketitle

% ----------------------------------------------------------------
\section{Introduction}

For a finitely generated group $G$, we denote the rank (i.e. the minimal number of the generators) of $G$ by $\rk G$. There are lots of research on the intersection of subgroups of a finitely generated group $G$ in the literature. For example, when $G$ is a free group, H. Neumann (see \cite{N1} and \cite{N2}) conjectured that for any two finitely generated subgroups $A$ and $B$ of $G$, $$\rk(A\cap B)-1 \leq (\rk A-1)(\rk B -1).$$
This conjecture was proved independently by I. Mineyev \cite{M} and J. Friedman \cite{F}.

Before this celebrated result was proved, it had been shown that for some special subgroups of free groups, people could say more about their intersection. Denote the set of endomorphisms of $G$ by $\edo(G)$. For a family $\mathcal B$ of endomorphisms of $G$, namely, $\mathcal B\subseteq \edo(G)$, the subgroup fixed by $\mathcal B$ is
$$\fix(\mathcal B):=\{g\in G|\phi(g)=g, \forall \phi\in \mathcal B\}.$$
It is called the $fixed$ $subgroup$ of $\mathcal B$. We abbreviate $\fix(\mathcal B)$ to $\fix\mathcal B$, and $\fix(\{\phi\})$ to $\fix\phi$ for any single endomorphism $\phi:G\to G$ in the context. It is obvious that $\fix\mathcal{B}=\bigcap_{\phi\in\mathcal{B}} \fix\phi$.

In \cite{BH}, M. Bestvina and M. Handel proved the Scott's conjecture that for any automorphism $\phi$ of a finitely generated free group $G$, $$\rk\fix\phi \leq \rk G.$$

In the book \cite{DV}, W. Dicks and E. Ventura generalized Bestvina-Handel's result of the fixed subgroup of a single automorphism to a family of injective endomorphisms. They proved the following theorem.

\begin{thm}\cite[Corollary IV.5.8]{DV}\label{Dicks-Ventura thm}
Let $G$ be a finitely generated free group, and $\mathcal B$ a family of injective endomorphisms of $G$. Then
$$\rk \fix\mathcal B\leq\rk G.$$
\end{thm}

They also showed that $\fix\mathcal B$ is inert in $G$ (see \cite[Theorem IV.5.7]{DV}). A subgroup $A$ is $inert$ in $G$ if for every subgroup $B\leq G$,
$$\rk(A\cap B)\leq \rk B.$$

In the paper \cite{B}, G. Bergman proved that Dicks-Ventura's result (Theorem \ref{Dicks-Ventura thm}) also holds for any family of endomorphisms but kept the following question open: Is the fixed subgroup of a family of endomorphisms of a finitely generated free group inert?

When $G$ is a $surface$ $group$, namely, $G$ is isomorphic to the fundamental group of some connected closed surface $S$ with Euler characteristic $\chi(S)<0$. In \cite{So1}, T. Soma  estimated the rank of the intersection of any two subgroups $A$ and $B$ of a surface group in terms of ranks of $A$ and $B$. In \cite{So2}, he showed an enhanced version of the result of \cite{So1}:
$$\rk(A\cap B)-1\leq 1161(\rk A-1)(\rk B-1).$$
In fact, since Hanna Neumann's Conjecture was proved, by \cite[Section 8]{M1}, T. Soma's result should be improved to be:
$$\rk(A\cap B)-1\leq (\rk A-1)(\rk B-1).$$

When $G$ is a one-relator group, D. Collins \cite{C}\cite{C2} studied the intersection of Magnus subgroups that we will describe in Section 4.

\vspace{6pt}

In this paper, we consider the intersection of the fixed subgroups of endomorphisms of a surface group and prove a theorem similar to that B. Bergman do on free group. We also consider the intersection of a subgroup with a geometric subgroup defined below that is similar but not the same as that of P. Scott \cite{S}.

A connected subsurface (i.e., two dimensional submanifold) $F$ of a connected surface $S$ is called $incompressible$ if the natural homomorphism $\pi_1(F)\to\pi_1(S)$ induced by the inclusion $F\hookrightarrow S$ is injective. If $F$ is incompressible in $S$, then we can think of $\pi_1(F)$ as a subgroup of $\pi_1(S)$. Subgroups which arise in this way are called $geometric$.

For the fixed subgroup of a single endomorphism of a surface group, B. Jiang, S. Wang and Q. Zhang \cite{JWZ} showed that

\begin{thm}\cite[Theorem 1.2]{JWZ}\label{JWZ thm}
Let $G$ be a surface group, and $\phi$ an endomorphism of $G$. Then

{\rm(1)} $\rk \fix \phi\leq \rk G$ if $\phi$ is epimorphic, with equality if and only if $\phi=id$;

{\rm(2)} $\rk \fix \phi\leq \frac{1}{2}\rk G$ if $\phi$ is not epimorphic.
\end{thm}

We generalize this result to any family of endomorphisms. The main result of this paper is that

\begin{thm}\label{main thm}
Let $G$ be a surface group, and $\mathcal B$ a family of endomorphisms of $G$. Then

{\rm(1)} $\rk \fix\mathcal B\leq \rk G$, with equality if and only if $\mathcal B=\{id\}$;

{\rm(2)} $\rk \fix\mathcal B\leq \frac{1}{2}\rk G$, if $\mathcal B$ contains a non-epimorphic endomorphism.
\end{thm}

For the geometric subgroups of a surface group, we prove that

\begin{thm} \label{main thm of inert}
If $A$ is a geometric subgroup of a surface group $G$, then $A$ is inert in $G$. Namely, for any subgroup $B$ of $G$, we have
$$\rk(A\cap B)\leq \rk B.$$
\end{thm}

As a corollary, we have

\begin{cor}\label{main corollary}
The fixed subgroup of any family of epimorphisms of a surface group $G$ is inert in $G$.
\end{cor}

The paper is organized as follows. In Section 2, we give some definitions and background knowledge for fixed point theory on surfaces and prove a strong version of Theorem \ref{JWZ thm}. In Section 3, we study the inertia of geometric subgroups of surface groups, and give the proofs of Theorem \ref{main thm of inert} and Corollary \ref{main corollary}. The technology used in this section is the covering theory. In Section 4, we discuss retracts and equalizers of a surface group. These special subgroups play a key role in the proof of our main result which we do in Section 5. At last, we give some examples and questions in Section 6.

%%-------------------------------------------------------------------------------------------------------------------------------------------------------------
%----------------------------------------------------------------------------------------------------------------------------------
\section{The Fixed subgroup of a single endomorphism}

For the fixed subgroup of any single endomorphism of a surface group, we have the following theorem that is not stated but can be obtained from the paper \cite{JWZ}.

\begin{thm}\label{aut fixed subgp are geometric}
Let $G$ be a surface group, $\phi$ a non-identity epimorphism. If $\fix \phi$ is not cyclic, then
$\fix \phi$ is a geometric free subgroup of $G$ with
$$\rk \fix\phi<\rk G.$$
\end{thm}

To prove Theorem \ref{aut fixed subgp are geometric}, we need to introduce some facts on fixed points and fixed subgroups of a selfmap of a space.

For a selfmap $f:X\rightarrow X$ of a connected compact polyhedron $X$, the fixed point set
$$\fix f:=\{x\in X|f(x)=x\}$$
splits into a disjoint union of $fixed$ $point$ $classes$: two fixed points are in the same class if and only if they can be joined by a $Nielsen~ path$, which is a path homotopic (rel. its endpoints) to its $f$-image. For each fixed point class $\F$, a homotopy invariant $index$ $\ind(f,\F)\in \Z$ is defined. A fixed point class is $essential$ if its index is non-zero, otherwise, called $inessential$ (see \cite{fp1} for an introduction).

Although there are several approaches to define fixed point classes, we state the one using paths and introduce another homotopy invariant: the $rank$ of a fixed point class $\F$ (see \cite[\S2]{JWZ}).

\begin{defn}
By an $f$-$route$ we mean a homotopy class (rel. endpoints) of path $w:I\rightarrow X$ from a point $x\in X$ to $f(x)$. For brevity we shall often say the path $w$ (in place of the path class $[ w ]$) is an $f$-route at $x=w(0)$. An $f$-route $w$ gives rise to an endomorphism
$$f_{w}:\pi_1(X,x) \rightarrow\pi_1(X,x),~~[ a] \mapsto [ w(f\comp a)\overline w ] $$
where $a$ is any loop based at $x$, and $\overline w$ denotes the reverse of $w$.

Two $f$-routes $[ w ]$ and $[ w' ]$ are $conjugate$ if there is a path $q:I\rightarrow X$ from $x=w(0)$ to $x'=w'(0)$ such that $[ w' ]=[ \overline qw (f\comp q)]$, that is, $w'$ and $\overline qw (f\comp q)$ are homotopic rel. endpoints.
\end{defn}

Note that a constant $f$-route $w$ corresponds to a fixed point $x=w(0)=w(1)$ of $f$, and the endomorphism $f_{w}$ becomes the natural homomorphism induced by $f$,
$$f_*:\pi_1(X,x)\rightarrow \pi_1(X,x),~~[ a]\mapsto  [ f\comp a],$$
where $a$ is any loop based at $x$. Two constant $f$-routes are conjugate if and only if the corresponding fixed points can be joint by a Nielsen path. This gives the following definition.

\begin{defn}
With an $f$-route $w$ (more precisely, with its conjugacy class) we associate a $fixed$ $point$ $class$ $\F_{w}$ of $f$, which consists of the fixed points that correspond to constant $f$-routes conjugate to $w$. Thus fixed point class are associated bijectively with conjugacy classes of $f$-routes. A fixed point class $\F_{w}$ can be empty if there is no constant $f$-route conjugate to $w$. Empty fixed point classes are inessential and distinguished by their associated route conjugacy classes.
\end{defn}

\begin{defn}
For any $f$-route $w$, the fixed subgroup of the endomorphism $f_w$ is the subgroup
$$\fix(f_w)=\{\gamma\in \pi_1(X,w(0))|f_w(\gamma)=\gamma\}.$$
Hence, we have the $rank$ of $\F_w$ defined as
$$\rk(f,\F_w):=\rk\fix(f_w),$$
it is well defined because conjugate $f$-routes have isomorphic fixed subgroups. Moreover, $\rk(f,\F_w)$ of a fixed point class is also a homotopy invariant.
\end{defn}

\vspace{6pt}

According to Nielsen-Thurston's canonical classification theorem of surface homeomorphisms, any homeomorphism of a compact connected surface with negative Euler characteristic is isotopic to either a periodic, pseudo-Anosov or reducible map $f$ (see W. Thurston \cite{T}). Moreover, B. Jiang and J. Guo \cite{JG} stated that such $f$ has a standard form so we call it a $standard$ $map$. A standard map has fine-tuned local behavior and nice properties.

By the complete list of possible types of fixed point classes of a standard map given in  \cite[Lemma 3.6]{JG}, we have

\begin{lem}\label{Subsurface}
Every fixed point class of a standard map of a closed surface with negative Euler characteristic is an incompressible compact connected subsurface (possibly a point or a circle).
\end{lem}

\begin{thm}\cite[Theorem 3.2]{JWZ}\label{fixed subgp of empty fixed pt class of standard map}
Let $f: S\to S$ be a standard map of a connected closed surface $S$ with $\chi(S)<0$. Then for any empty fixed point class $\F$, we have
$$\rk(f,\F)\leq 1.$$
\end{thm}

\vspace{6pt}

Now we give the proof of Theorem \ref{aut fixed subgp are geometric}.

\begin{proof}[Proof of Theorem \ref{aut fixed subgp are geometric}]

Let $G$ be a surface group, namely, $G=\pi_1(S)$ for some closed surface $S$ with $\chi(S)<0$. Then it is obvious that $\rk \fix \phi<\rk G$ by Theorem \ref{JWZ thm}. Now we show that $\fix\phi$ is a geometric free subgroup.

Note that $S$ is a $K(G, 1)$ space, then the endomorphism $\phi: G\to G$ is induced by a selfmap $g: (S,x)\to(S,x)$ (see \cite[Proposition 1B.9]{Ha}). Namely,
$$\phi=g_*: \pi_1(S,x)\to \pi_1(S,x),~~[a]\mapsto [g\comp a],$$
where $a$ is any loop based at $x\in\fix g$.

Since $\phi$ is epimorphic, it is an automorphism because $G$ is Hopfian. Hence, $g$ can be homotopic to a homeomorphism, even to a standard map $f:(S,x)\to (S,f(x))$, via a homotopy $H=\{h_t\}_{t\in I}$. Then
$$\phi=f_w: \pi_1(S,x)\to \pi_1(S,x),~~[a]\mapsto [w(f\comp a)\bar w],$$
where $w=\{h_t(x)\}_{{t\in I}}$. Therefore, $\fix \phi=\fix(f_w)$.

Note that $\fix(f_w)$ is not cyclic. Then the fixed point class $\F_w$ corresponding to $w$ is not empty according to Theorem \ref{fixed subgp of empty fixed pt class of standard map}. Thus there is a fixed point $*\in \F_w\subseteq \fix f$ that is conjugate to $w$, namely, the loop $\bar qw(f\comp q)$ is homotopic to the point $*$, where $q$ is a path from $x=w(0)$ to $*$. We have the following commutative diagram
$$\xymatrix{
\pi_1(S,x)\ar[d]^{\cong}_{q_{\sharp}}\ar[r]^{f_{w}}& \pi_1(S,x)\ar[d]^{q_{\sharp}}_{\cong}\\
\pi_1(S,*)\ar[r]^{f_*} &\pi_1(S,*)
}$$
where $q_{\sharp}:[a]\mapsto [\bar qaq]$ is an isomorphism. Therefore, under the isomorphism $q_{\sharp}$, we can pick a new presentation $G=\pi_1(S,*)$, and \begin{equation}\label{eq.5}
\phi=f_*:\pi_1(S,*)\to\pi_1(S,*),~~[a]\mapsto [f\comp a],
\end{equation}
where $a$ is any loop based at $*$.

Recall that $f$ is a standard map, then each Nielsen path of $f$ can be deformed (rel. endpoints) into $\fix f$ by \cite[Proof of Corollary T]{JWZ} or \cite[Lemmas 1.2, 2.2 and 3.4]{JG}. Hence every fixed point class is connected, and
$$\fix (f_*)=\pi_1(\F_w, *)\leq \pi_1(S,*),$$
the last inequality holds because the fixed point class $\F_w$ is an incompressible subsurface according to Lemma \ref{Subsurface}. Therefore, the fixed subgroup $\fix (f_*)$ is geometric. Clearly, $\F_w$ is a compact subsurface with nonempty boundary because $\phi\neq id$, hence $\fix (f_*)$ is free. Therefore, $\fix \phi$ is a geometric free subgroup of $G$ by equation (\ref{eq.5}).
\end{proof}

At the end of this section, we give a lemma used frequently in this paper.

\begin{lem}\label{subgp of surface gp}
Let $H$ be a proper subgroup of a surface group $G$ with $\rk H\leq \rk G$. Then $H$ is a free group. Furthermore, if $\phi:G\to G$ is an endomorphism but non-epimorphic, then $\phi(G)$ is a free group with
$$\rk \phi(G)\leq \frac{1}{2}\rk G.$$
\end{lem}

\begin{proof}
Let $G=\pi_1(S)$, where $S$ is a closed surface with $\chi(S)<0$. By covering theory of surfaces, the proper subgroup $H$ is either free or $H\cong \pi_1(\widetilde S)$ for some closed surface $\widetilde S$ with $\chi(\widetilde S)/\chi(S)=|G:H|>1$. But the latter implies
$\rk \pi_1(\widetilde S)>\rk\pi_1(S)$ which contradicts to $\rk H\leq \rk G$. Therefore, $H$ is a free group.

If $\phi:G\to G$ is an endomorphism but non-epimorphic, then $\phi(G)<G$ and $\rk\phi(G)\leq \rk G$. Thus $\phi(G)$ is a free group. Moreover, we have
$$\rk \phi(G)\leq \ir(G),$$
where $\ir(G)$ denotes the $inner$ $rank$ of $G$ defined as the maximal rank of free
homomorphic images of $G$. It is known that when $G$ is a surface group, then $\ir(G)=[\frac{1}{2}\rk G]$, the greatest integer not more than $\frac{1}{2}\rk G$. (See Lyndon and Schupp \cite[page 52]{LS} where it is attributed to Zieschang \cite{Z}.) Therefore, we have $\rk \phi(G)\leq \frac{1}{2}\rk G$.
\end{proof}

%------------------------------------------------------------------------------------------------------------------------------------------------------------------
\section{The inertia of geometric subgroups of surface groups}

In this section, we study the inertia of geometric subgroups of surface groups firstly, and then give the proofs of Theorem \ref{main thm of inert} and Corollary \ref{main corollary}.

\subsection{The inertia of geometric subgroups of surface groups}

For the intersection of a subgroup with a geometric subgroup in the fundamental group of a surface, we have

\begin{thm} \label{inert of geometric subgp}
Let $S$ be a connected surface (may has punctures) with $\pi_1(S)$ finitely generated. If the subgroup $A\leq \pi_1(S)$ is geometric, and $B$ is any subgroup of $\pi_1(S)$, then
$$\rk(A\cap B)\leq \rk B.$$
\end{thm}

\vspace{6pt}
Before giving the proof of Theorem \ref{inert of geometric subgp}, we need the lemma below. For brevity, a subsurface means it is connected unless it is specially stated otherwise.

\begin{lem}\label{subsurf. gp rank}
Suppose $S$ is a connected surface with $\pi_1(S)$ finitely generated, and $F\subseteq S$ is an incompressible subsurface. Then $\rk \pi_1(F)\leq \rk\pi_1(S)$. In particular, if $S$ is closed, then $\pi_1(F)$ is either $\pi_1(S)$ itself or a free group with
$$\rk \pi_1(F)<\rk \pi_1(S).$$
\end{lem}

\begin{proof}
If $F$ is closed, then $S$ must be closed and $\pi_1(F)=\pi_1(S)$. Now we suppose $F$ is neither closed nor a disk, then $\pi_1(F)$ is a free group. There are two cases:

Case (1).  Both $F$ and $S$ are compact. Via a slightly push of $\partial F$ into $\intt S$, we can assume that $\partial F\cap \partial S=\emptyset$. Then each component of $\partial (S-\intt F)$ is a circle which is either contained in $\partial F$ or $\partial S$. Recall that $F$ is incompressible and not a disk, then no component of $S-\intt F$ is a disk. In fact, if there is a disk $D$, then $\partial D\subseteq \partial F$ which contradicts to that the natural homomorphism $\pi_1(F)\to \pi_1(S)$ is injective. Thus the Euler characteristic $\chi(S-\intt F)\leq 0$ and $$\chi(S)=\chi(F)+\chi(S-\intt F)\leq \chi (F).$$
Therefore, $\rk \pi_1(F)\leq \rk\pi_1(S)$ when $S$ is not closed, and $\rk \pi_1(F)<\rk\pi_1(S)$ when $S$ is closed.

Case (2).  At least one of $F$ and $S$ is not compact. If $F$ is not compact, pick a core $C_F$ of $F$, which is a compact subsurface of $F$ such that each component of $F-C_F$ is an open annulus; if $F$ is compact, set $C_F=F$. Moreover, we can choose a compact core $C_S$ of $S$ such that $C_F\subseteq C_S$ (see \cite[Lemma 1.5]{S}). Thus $\pi_1(C_F)=\pi_1(F)$, $\pi_1(C_K)=\pi_1(K)$ and the natural homomorphism $\pi_1(C_F)\to \pi_1(C_S)$ is also injective. The conclusion holds by Case (1).
\end{proof}

\begin{proof}[Proof of Theorem \ref{inert of geometric subgp}]
If $B$ is infinitely generated (i.e., $\rk B=\infty$) or $A\cap B=\{1\}$, it is trivial. Thus we assume $B$ is finitely generated and $A\cap B\neq \{1\}$ below.

Since $A$ is geometric, there is an incompressible subsurface $F\subseteq S$ such that $A=\pi_1(F,*)\leq \pi_1(S,*)$ for some base point $*\in F$. We have two maps:
the inclusion $i: (F,*)\hookrightarrow (S,*)$, and the covering $k: (K,*_K)\to (S,*)$ associated to $B$, namely, $k_*(\pi_1(K,*_K))=B$.
Consider the commutative diagram:
$$\xymatrix{
(F_0,*_0)\subseteq (\tilde F,*_0)\ar[d]_{p'}\ar[r]^<<<<<{i'}\ar[dr]^p& (K,*_K)\ar[d]^k\\
(F,*)\ar@{^{(}->}[r]^i &(S,*)
}$$
where
$$\tilde F=\{(x,y)\in F\times K|i(x)=k(y)\},$$
$p: \tilde F\to S$ is the pull back map such that
$$p((x,y))=i(x)=k(y),$$
and $F_0$ is the component of $\tilde F$ containing the base point $*_0=(*,*_K)$. Since $i:F\hookrightarrow S$ is an inclusion, $\tilde F$ can be identified as $k^{-1}(F)$, and $p':\tilde F\to F$ can be identified as the covering $k|_{k^{-1}(F)}:k^{-1}(F)\to F$. Thus $F_0$ is a connected subsurface of $K$, and by the commutative diagram, $i'_*:\pi_1(F_0)\to \pi_1(K)$ is injective since $p': F_0\to F$ is a covering.
Therefore, $F_0$ is incompressible in $K$. By Lemma \ref{subsurf. gp rank}, we have
\begin{equation}\label{eq.1}
\rk\pi_1(F_0)\leq \rk \pi_1(K)=\rk B.
\end{equation}
Moreover, $p_*:\pi_1(F_0)\to \pi_1(S)$ is also injective according to the commutative diagram, and we have
$$p_*(\pi_1(F_0))\leq i_*(\pi_1(F))\cap k_*(\pi_1(K))=A\cap B.$$

Now we claim that
\begin{equation}\label{eq.2}
p_*(\pi_1(F_0))=A\cap B.
\end{equation}

To prove the claim, it suffices to prove $A\cap B\leq p_*(\pi_1(F_0))$. In fact, any nontrivial element $a\in A\cap B$ can be represented by a loop $\alpha \subset F\subseteq S$ containing the base point $*$. Since $[\alpha]=a\in A\cap B$, there is a lifting loop $\tilde \alpha \subset K$ containing the base point $*_K$. Therefore there is a loop $\tilde \alpha_0\subset F_0$ containing the base point $*_0$ and $p(\tilde \alpha_0)=\alpha$, which implies $A\cap B\leq p_*(\pi_1(F_0))$. Thus the claim holds.

Therefor we have
$$\rk(A\cap B)=\rk p_*(\pi_1(F_0))=\rk \pi_1(F_0)\leq \rk B,$$
where the first equality holds by equation (\ref{eq.2}), the second equality holds because $p_*$ is injective, and the last inequality holds by equation (\ref{eq.1}).
\end{proof}

\subsection{Proofs of Theorem \ref{main thm of inert} and Corollary \ref{main corollary}}

It is obvious that Theorem \ref{main thm of inert} follows from Theorem \ref{inert of geometric subgp}. Now we give the proof of Corollary \ref{main corollary}.

\begin{proof}[Proof of Corollary \ref{main corollary}]
Let $G$ be a surface group, namely, $G\cong \pi_1(S)$ for a closed surface $S$ with $\chi(S)<0$, and $\mathcal B$ a family of epimorphisms of $G$. Suppose $K\leq G$ is any subgroup of $G$.

If $\mathcal B=\{id\}$, or some $\beta\in \mathcal B$ has $\fix \beta$ cyclic, then $\rk(K\cap\fix\mathcal B)\leq \rk K$ is obvious.

Now we suppose $\mathcal B\neq\{id\}$ and $\fix \beta$ is not cyclic for all $\beta\in \mathcal B$. Then by Theorem \ref{aut fixed subgp are geometric}, for any non-identity epimorphism $\beta\in \mathcal B$, $\fix\beta$ is a geometric free subgroup with $\rk \fix\beta<\rk G$. Hence without loss of generality, we assume $id\not\in \mathcal B$ in the following.

Note that $G$ is finitely generated, then $\edo(G)$, the set of all endomorphisms of $G$, is countable and hence $\mathcal B\subseteq \edo(G)$ is also countable. Set $\mathcal B=\{\beta_1,\beta_2,\ldots\}$ and $\mathcal B_i=\{\beta_1,\ldots,\beta_i\}, i=1,2,\ldots$ (If $\mathcal B$ has a finite cardinality $n$, set $\beta_j=\beta_n$ and $\mathcal B_j =\mathcal B_n$ for all $j>n$). Then we have a descending chain of fixed subgroups
$$\fix\beta_1=\fix \mathcal B_1\geq \fix \mathcal B_2\geq\ldots\geq\fix \mathcal B_i\geq\ldots$$
whose terms are all free groups. Furthermore, we have a descending chain of free groups
$$K\cap\fix\beta_1=K\cap\fix \mathcal B_1\geq K\cap \fix \mathcal B_2\geq\ldots\geq K\cap\fix \mathcal B_i\geq\ldots.$$
Note that $K\cap\fix \mathcal B_{i+1}=K\cap\fix \mathcal B_i\cap \fix \beta_{i+1}$ and $\fix\beta_{i+1}$ is geometric in $G$ for all $i\geq 1$, by Theorem \ref{main thm of inert}, we have
$\rk(K\cap\fix \mathcal B_{i+1})\leq\rk(K\cap\fix \mathcal B_i)$.
Thus
$$\rk(K\cap\fix\mathcal \beta_1)=\rk(K\cap\fix \mathcal B_1)\geq \rk(K\cap\fix \mathcal B_2)\geq\ldots\geq\rk(K\cap\fix \mathcal B_i)\geq\ldots.$$
Note that
$$K\cap\fix\mathcal B=K\cap(\bigcap_{i=1}^{\infty} \fix\mathcal B_i)=\bigcap_{i=1}^{\infty} (K\cap\fix\mathcal B_i),$$
we have
$$\rk (K\cap\fix\mathcal B)=\rk(\bigcap_{i=1}^{\infty} (K\cap\fix\mathcal B_i))\leq \rk(K\cap\fix\beta_1)\leq\rk K$$
where the first inequality is according to \cite[Exercise 33, p.118]{MKS} that if the intersection of a descending chain of free groups has rank $\geq n$, then almost all terms of the chain have rank $\geq n$, and the second inequality holds since $\fix\beta_1$ is geometric in $G$. It implies that $\fix\mathcal B$ is inert.
\end{proof}

%------------------------------------------------------------------------------------------------------------------------------------------------------------------
\section{Equalizers and retracts}

In this section, we study the equalizers and retracts of surface groups.

\subsection{Introduction to equalizers and retracts}

Suppose $G$ and $H$ are two groups, $\phi: G\to H$ is an epimorphism. A $section$ of $\phi$ is a homomorphism $\sigma: H\to G$ such that
$$\phi\comp \sigma=id: H\to H.$$
Then for any family $\mathcal{B}$ of sections of $\phi$, the $equalizer$ of $\mathcal B$
$$\eq(\mathcal B):=\{h\in H|\sigma_1(h)=\sigma_2(h), \forall \sigma_1,\sigma_2\in \mathcal B\}$$
is a subgroup of $H$.

Suppose $H$ is a subgroup of a group $G$. If there is a homomorphism $\pi: G\to G$ such that $\pi(G)\leq H$ and
$$\pi|_H=id: H\to H,$$
we say that $\pi$ is a $retraction$, and $H$ is a $retract$ of $G$. We have $\rk H\leq\rk G$ obviously. Moreover, if $H$ is a proper subgroup, it is called a $proper$ $retract$. Note that if a retract $H$ of $G$ is contained in a subgroup $K\leq G$, then $H$ is also a retract of $K$. Hence $\rk H\leq \rk K$.

\vspace{6pt}

The following is a relation between equalizers and retracts.

\begin{lem}\label{equalizer & retract}
Let $G, H$ be two groups and $\phi:G\to H$ an epimorphism. If $\mathcal B$ is a family of sections of $\phi$. Then for any section $\sigma\in \mathcal B$, $\sigma(H)$ is a retract of $G$, and
$$\sigma|_{\eq(\mathcal B)}: \eq(\mathcal B)\to \bigcap_{\alpha\in \mathcal B}\alpha (H)$$
is an isomorphism.
\end{lem}

\begin{proof}
For any $\sigma\in \mathcal B$, $\sigma: H\to G$ is a section of $\phi: G\to H$ implies
$$\phi\comp\sigma=id: H\to H,$$
and hence $\sigma\comp \phi: G\to \sigma(H)$
is an epimorphism such that
$$(\sigma\comp \phi)(\sigma (h))=\sigma(\phi(\sigma(h)))=\sigma(h)$$
for any $h\in H$. Therefore, $\sigma(H)$ is a retract of $G$.

Clearly $\sigma$ is injective and $\sigma(\eq(\mathcal B))\leq \bigcap_{\alpha\in \mathcal B}\alpha (H)$. To prove $\sigma|_{\eq(\mathcal B)}$ is an isomorphism, it suffices to show
$$\sigma(\eq(\mathcal B))= \bigcap_{\alpha\in \mathcal B}\alpha (H).$$
In fact, for any $g\in  \bigcap_{\alpha\in \mathcal B}\alpha (H)$ and any $\alpha\in\mathcal B$, there exists $h_{\alpha}\in H$ such that $g=\alpha(h_{\alpha})$.
Hence $\phi(g)=\phi(\alpha(h_{\alpha})=h_{\alpha}$, which implies $\phi(g)\in \eq(\mathcal B)$ and hence $g=\sigma(\phi(g))\in \sigma(\eq(\mathcal B))$.
\end{proof}

For equalizers and retracts of finitely generated free groups, G. Bergman gave the following results (see \cite[Corallary 12 and Lemma 18]{B}).

\begin{prop}[Bergman]\label{intersection of retracts of free gp}
{\rm(1)} Any intersection of retracts of a finitely generated free group is also a retract;

{\rm(2)} If $\phi: G\to H$ is an epimorphism of free groups with $H$ finitely generated, then the equalizer of any family of sections of $\phi$ is a free factor in $H$.
\end{prop}

\subsection{Some facts on surface groups}

In this subsection, we introduce some facts on surface groups.

Let $S$ be a closed surface of genus $g$. It is well known that $\pi_1(S)$ has a standard presentation:
$$\pi_1(S)=\langle a_1,b_1,\ldots,a_g,b_g|\prod_{i=1}^g [a_i, b_i]\rangle,~\quad or \quad \pi_1(S)=\langle a_1,a_2,\ldots,a_g|\prod_{i=1}^g a_i^2\rangle$$
according to whether $S$ is orientable or not.

For any generating set $X=\{x_1,\ldots, x_n\}\subset \pi_1(S)$, let $F_X = \langle y_1, . . . , y_n\rangle$ be a free group with one generator for each element of $X$ and denote the natural map $F_X\to \pi_1(S)$ by $\phi_X$. Two generating sets $X$ and $X'$ of the same cardinality are $Nielsen$ $equivalent$
if there is an isomorphism $\epsilon: F_{X'}\to F_X$ such that the following diagram commutes.
$$
\xymatrix{
F_{X'}  \ar[rr]^{\epsilon} \ar[dr]_{\phi_{X'}}
                &  &    F_X \ar[dl]^{\phi_X}    \\
                & \pi_1(S)                 }
$$
%$\phi_X=\phi_{X'}\comp \epsilon$.

In the paper \cite{Z2}, H. Zieschang showed that for any generating set $X\subset\pi_1(S)$ with cardinality $|X|=\rk \pi_1(S)$ for a closed orientable surface $S$ of genus not 3 is Nielsen equivalent to the standard generating set. In \cite{L}, L. Louder generalized the result to any closed surface of any genus whether it is orientable or not.

For standard generating set $X'$ of $\pi_1(S)$, the kernel of $\phi_{X'}$ is the normal closure of a word $w$ in $F_{X'}$. Thus for any generating set $X$ of $\pi_1(S)$ with $|X|=\rk \pi_1(S)$,  the kernel of $\phi_{X}$ is the normal closure of a word $r=\epsilon(w)$ in $F_{X}$. Moreover, we can let $r$ be cyclically reduced. Hence we have

\begin{lem}\label{Nielsen equav. of surface group}
Let $S$ be a closed surface, and $n=\rk\pi_1(S)$. If $X=\{x_1, x_2,\ldots, x_n\}$ is any generating set of $\pi_1(S)$, then $\pi_1(S)$ has a new one-relator presentation
$$\pi_1(S)=\langle x_1, x_2,\ldots, x_n|r \rangle,$$
where $r$ is a cyclically reduced word in the free group on the generating set $X$.
\end{lem}

Let $G=\langle X|r\rangle$ be a one-relator group where $r$ is a cyclically reduced word in the free group on the generating set $X$. A subset $Y\subset X$
is called a $Magnus$ $subset$ if $Y$ omits a generator which appears in the relator $r$. A subgroup $H$ of $G$ is called a
$Magnus$ $subgroup$ if $H=\langle Y\rangle$ for some Magnus subset $Y$ of $X$, and hence by the Magnus Freiheitssatz \cite[Theorem 4.10]{MKS}, $H$ is free of rank $|Y|$. There were many studies (\cite{Br}\cite{C}\cite{C2}\cite{Ho}, etc) on intersections of Magnus subgroups. In particular, D. Collins showed that

\begin{thm}\cite[Theorem 2]{C}\label{intersection of Mugnus subgp}
The intersection $\langle Y \rangle \cap\langle Z\rangle$ of two Magnus subgroups of the
one-relator group $G$ is either $\langle Y \cap Z\rangle$ or the free product of $\langle Y \cap Z\rangle$ with
an infinite cyclic group and thus of rank $|Y \cap Z| + 1$.
\end{thm}

\subsection{Equalizers and retracts on surface groups}

Now we consider equalizers and retracts of surface groups, which will play a key role in the proof of Theorem \ref{main thm}.

\begin{lem}\label{intersection of two retracts}
Let $G$ be a surface group. If $K$ is any proper retract of $G$, then $K$ is a free group with rank
$$\rk K\leq \frac{1}{2}\rk G.$$

Furthermore, if $H_1,H_2$ are two proper retracts of $G$, and $H=\langle H_1, H_2\rangle\leq G$, the subgroup generated by $H_1$ and $H_2$, then

{\rm(1)} If $H<G$, then $H$ is a free group, $H_1\cap H_2$ is a retract of both $H_1$ and $H_2$, and
$$\rk(H_1\cap H_2)\leq \min\{\rk H_1,\rk H_2\}.$$

{\rm(2)} If $H=G$, then $H_1\cap H_2$ is cyclic (possibly trivial).
\end{lem}

\begin{proof}
Since $K$ is a proper retract of the surface group $G$, there is an endomorphism $\beta: G\to G$ such that $\beta(G)=K<G$ and $\beta|_K=id$. By Lemma \ref{subgp of surface gp}, $K$ is a free group with $\rk K\leq \frac{1}{2}\rk G$.

Furthermore,  since $H_1$ and $H_2$ are two proper retracts of $G$ and $H=\langle H_1, H_2\rangle$, we have
\begin{equation}\label{eq.6}
\rk H\leq \rk H_1+\rk H_2\leq \rk G.
\end{equation}

There are two cases.

Case (1).  $H<G$. Then $H$ is a free group by Lemma \ref{subgp of surface gp}. Note that $H_1$ and $H_2$ are both retracts of the free group $H$, then $H_1\cap H_2$ is also a retract of $H$ according to Proposition \ref{intersection of retracts of free gp}. It implies
$$\rk(H_1\cap H_2)\leq \min\{\rk H_1,\rk H_2\}.$$

Case (2).  $H=G$. Let $X_i$ be a generating set of $H_i$, $i=1,2$. Then $X_1\cup X_2$ is a generating set of $G$, moreover $|X_1\cup X_2|=\rk G$ and $X_1\cap X_2=\emptyset$ by equation (\ref{eq.6}). Thus by Lemma \ref{Nielsen equav. of surface group}, $G$ has a one-relator presentation
$$G=\langle X_1\cup X_2|r\rangle$$
where $r$ is a cyclic reduced word in the free group on the generating set $X_1\cup X_2$. It implies that both $X_1$ and $X_2$ are Magnus subset and hence $H_1$ and $H_2$ are both Magnus subgroup of $G$. Therefore, the intersection $H_1\cap H_2$ is a cyclic (possibly trivial) subgroup of $G$ according to Theorem \ref{intersection of Mugnus subgp}.
\end{proof}

\begin{prop}\label{intersection of a family of retracts of surface gp}
Let $G$ be a surface group and $\mathcal R$ a family of retracts of $G$. Then
$$\rk(\bigcap_{H\in \mathcal R}H)\leq \min\{\rk H|H\in \mathcal R\}\leq \left\{\begin{array}{cc}
\rk G,     & \mathcal R=\{G\} \\
\frac{1}{2}\rk G, & \mathcal R\neq \{G\}
\end{array}\right..\nonumber$$
\end{prop}

\begin{proof}
For any proper retract $H\in \mathcal R$, $H$ is a free group with rank
$$\rk H\leq \frac{1}{2}\rk G<\rk G$$
according to Lemma \ref{intersection of two retracts}. Therefore, it suffices to assume that $\mathcal R$ consists of proper retracts in the following.
There are two cases.

Case (1).  There exist two retracts $H, H'\in \mathcal R$ such that $G=\langle H,H'\rangle$, then $H\cap H'$ is cyclic by Lemma \ref{intersection of two retracts}. Note that $\bigcap_{H\in \mathcal R}H$ is a subgroup of the cyclic group $H\cap H'$, we have $\bigcap_{H\in \mathcal R}H$ is also cyclic, which implies
$$\rk(\bigcap_{H\in \mathcal R}H)\leq 1\leq \min\{\rk H|H\in\mathcal R\}.$$

Case (2).  For any two retracts $H,H'$, $\langle H,H'\rangle<G$. Let $H_0\in \mathcal R$ be the retract which has the minimal rank in $\mathcal R$, namely,
$$\rk H_0=\min\{\rk H|H\in \mathcal R\}.$$
By Lemma \ref{intersection of two retracts}, $\{H_0\cap H|H\in \mathcal R\}$ is a family of retracts of the free group $H_0$. Therefore
$$\bigcap_{H\in \mathcal R}H=H_0\cap(\bigcap_{H\in \mathcal R} H)=\bigcap_{H\in \mathcal R}(H_0\cap H)$$
is a retract of $H_0$ according to Proposition \ref{intersection of retracts of free gp}. Hence
$$\rk(\bigcap_{H\in \mathcal R}H)\leq \rk H_0\leq \min\{\rk H|H\in \mathcal R\}.$$
The proof is finished.
\end{proof}

\begin{prop}\label{equalizer of endomorphisms of surface group}
Let $G$ be a surface group and $F$ a finitely generated free group. If $\phi: G\to F$ is an epimorphism, and $\mathcal B$ is a family of sections of $\phi$, then
$$\rk\eq(\mathcal B)\leq \rk F\leq\frac{1}{2}\rk G.$$
\end{prop}

\begin{proof}
For any section $\sigma\in \mathcal B$, $\sigma(F)<G$ since $\sigma(F)$ is isomorphic to the free group $F$. By Lemma \ref{equalizer & retract}, we have an isomorphism
$$\sigma|_{\eq(\mathcal B)}: \eq(\mathcal B)\to \bigcap_{\alpha\in \mathcal B}\alpha(F)$$
where $\{\alpha(F)|\alpha\in \mathcal B\}$ is a family of proper retracts of $G$. Therefore, the conclusion holds according to Proposition \ref{intersection of a family of retracts of surface gp}.
\end{proof}

%----------------------------------------------------------------------------------------------------------------------------------
\section{Proof of Theorem \ref{main thm}}

The aim of this section is to prove Theorem \ref{main thm}.

\begin{proof}[\textbf{Proof of Theorem \ref{main thm}}]
Suppose $G$ is a surface group and $\mathcal B$ is a family of endomorphisms of $G$. There are two cases.

\vspace{6pt}
Case (1).  $\mathcal B$ consists of epimorphisms. Then $\fix \mathcal B$ is inert in $G$ by Corollary \ref{main corollary}, and we have
$$\rk \fix \mathcal B=\rk (G\cap\fix \mathcal B)\leq \rk G,$$
and when $\mathcal B=\{id\}$, the equality holds obviously. Moreover, if $\mathcal B\neq \{id\}$, then there is a non-identity epimorphism $\beta\in \mathcal B$ with $\rk\fix \beta<\rk G$ according to Theorem \ref{aut fixed subgp are geometric}, and we have
$$\rk \fix \mathcal B=\rk(\fix \beta\cap\fix \mathcal B)\leq \rk \beta< \rk G.$$

\vspace{6pt}

Case (2).  $\mathcal B$ contains a non-epimorphic endomorphism.

The proof of this case is partly inspired by G. Bergman's paper \cite{B}. Without loss of generality, we assume that $\mathcal B$ is closed under composition and contains the identity endomorphism.
Recall that $\mathcal B$ contains a non-epimorphic endomorphism, we can choose $\beta\in \mathcal B$ such that $\beta(G)$ is a free group with
$$\rk(\beta(G))=\min\{\rk(\gamma(G))|\gamma\in\mathcal B\}\leq \frac{1}{2}\rk G$$
according to Lemma \ref{subgp of surface gp}. Thus all elements of $\mathcal B$ act injective on $\beta (G)$. Indeed, if there is $\gamma\in \mathcal B$ acts not injective on $\beta (G)$, then $\rk(\gamma\beta(G))<\rk(\beta(G))$ contradicts to the minimality of $\rk(\beta(G))$. Let $\beta\mathcal B=\{\beta\gamma|\gamma\in \mathcal B\}$. Note that $\beta\gamma(\beta(G))\leq \beta(G)$, thus we have a family $\beta\mathcal B|_{\beta(G)}$ of injective endomorphisms of the free group $\beta(G)$,
$$\beta\gamma|_{\beta(G)}: \beta(G)\to \beta(G).$$
Since $\fix(\beta\mathcal B)=\fix(\beta\mathcal B|_{\beta(G)})\leq \beta(G)$, for brevity, we omit the restriction if no confusion is possible. Therefore, by Theorem \ref{Dicks-Ventura thm}, we have
\begin{equation}\label{eq.3}
\rk\fix(\beta\mathcal B)\leq \rk(\beta(G))\leq\frac{1}{2}\rk G.
\end{equation}

Clearly, $\fix\mathcal B$ is a subgroup of the free group $\fix(\beta\mathcal B)$. Now we claim that

\begin{claim}\label{claim}
$\rk\fix\mathcal B\leq\rk\fix(\beta\mathcal B)$.
\end{claim}

\noindent\emph{Proof.}
 Let
$$E=\beta^{-1}(\fix(\beta \mathcal B))\leq G,$$
then there is an epimorphism
$$\beta: E\to \fix(\beta \mathcal B),$$
and a family of sections $\mathcal B|_{\fix(\beta \mathcal B)}$ of $\beta$
$$\gamma|_{\fix(\beta \mathcal B)}:\fix(\beta \mathcal B)\to E,~~ \forall \gamma\in \mathcal B.$$
Note that $\fix \mathcal B\leq \fix(\beta \mathcal B)$ and $\mathcal B$ contains the identity (and hence $\mathcal B|_{\fix(\beta \mathcal B)}$ contains the identity), then
\begin{equation}\label{eq.4}
\fix\mathcal B=\fix(\mathcal B|_{\fix(\beta \mathcal B)})=\eq(\mathcal B|_{\fix(\beta \mathcal B)}).
\end{equation}

Recall that $E$ is a subgroup of the surface group $G$, then $E$ is either free or isomorphic to a surface group.
If $E$ is a free group, then by Proposition \ref{intersection of retracts of free gp}, $\eq(\mathcal B|_{\fix(\beta \mathcal B)})$ is a free factor of $\fix(\beta \mathcal B)$, and hence $\rk\eq(\mathcal B|_{\fix(\beta \mathcal B)})\leq \rk\fix(\beta \mathcal B)$; if $E$ is a surface group, then we also have $\rk\eq(\mathcal B|_{\fix(\beta \mathcal B)})\leq \rk\fix(\beta \mathcal B)$ according to Proposition \ref{equalizer of endomorphisms of surface group}. Thus by equation (\ref{eq.4}), Claim \ref{claim} holds.

\vspace{6pt}
Therefore, by equation (\ref{eq.3}) and Claim \ref{claim}, we have
$$\rk\fix\mathcal B\leq\frac{1}{2}\rk G$$
and the proof is finished.
\end{proof}

%--------------------------------------------------------------------------------------------------------------------------------------------
\section{Examples and Questions}

In this section, we give some examples and questions on surface groups.

The example below shows that the fundamental group of a torus also satisfies the conclusion of Theorem \ref{main thm}.

\begin{exam}
Let $G=\langle a,b|a^{-1}b^{-1}ab\rangle\cong \Z\oplus \Z$, and $\phi$ an endomorphism of $G$. It is well known that any subgroup of $G$ is also abelian with rank $\leq 2$. Thus
$$\rk \fix\phi\leq \rk G.$$

Now we claim that $\fix\phi$ is a cyclic group (possibly trivial) when $\phi\neq id$.

Indeed, pick a basis $a=(1,0)$ and $b=(0,1)$ of $G$. Then $G=\{(u,v)|u,v\in \Z\}$, and
$\phi$ can be presented as a $2\times 2$ matrix $A$ with integral entries
$$\phi(x)=xA,\quad  \forall x=(u,v)\in G.$$
If $\rk\fix\phi=2$, then there are two non-parallel vectors $x_1, x_2\in \fix\phi$ such that $x_1A=x_1$ and $x_2A=x_2$. For any $x\in G$, suppose
$x=kx_1+lx_2$, $k,l\in \Q$, it implies
$$xA=(kx_1+lx_2)A=kx_1A+lx_2A=kx_1+lx_2=x.$$
Namely, $\phi=id$. Therefore, the claim holds.
\end{exam}

The example below shows that the fundamental group of a Klein bottle has a nonidentity automorphism with fixed subgroup of rank 2, hence it does not satisfy the conclusion of Theorem \ref{main thm}.

\begin{exam}
Let $G=\langle a,b|b a b^{-1}a\rangle$ be the fundamental group of a Klein bottle, and $\phi$ an endomorphism of $G$. Since $a^{-1}b^{\epsilon}=b^{\epsilon}a, ab^{\epsilon}=b^{\epsilon}a^{-1}$ in $G$ for $\epsilon=\pm 1$, any element $g$ of $G$ can be write uniquely as $b^m a^n$. Suppose
$$\phi(a)=b^s a^t,\quad \phi(b)=b^p a^q.$$
We have
$$\phi(b a b^{-1}a)=b^p a^q b^s a^t a^{-q} b^{-p} b^s a^t= b^{2s} a^{(-1)^{s-p}[(-1)^s q+t-q]+t}=1.$$
Thus $s=0$ and $(-1)^p t+t=0$. There are two cases.

Case (1).  $t=0$. Then $\phi(a)=1, \phi(b)=b^p a^q$.

If there exists $1\neq g=b^m a^n\in G$ fixed by $\phi$, then
$$b^m a^n=g=\phi(g)=\phi(b^m a^n)=\phi(b^m)=(b^p a^q)^m=b^{mp} a^k.$$
We have $m=0$ or $p=1$. If $m=0$, then $g=a^n$ and $g=\phi(g)=\phi(a^n)=1$ contradicts to that $g$ is nontrivial.

So $p=1$, namely $\phi(a)=1, \phi(b)=b a^q$. $\fix\phi$ is a cyclic subgroup generated by $b a^q$.

Case (2).  $t\neq 0$ and $p$ is odd. Then $\phi(a)=a^t,\phi(b)= b^p a^q$.

If there exists $1\neq g=b^m a^n\in G$ fixed by $\phi$, then a same argument as in Case (1) implies $m=0$ or $p=1$. There are two subcases.

Subcase (2.1). If $p\neq 1$ which means $m=0$ , then $g=a^n\neq 1$ and $g=\phi(g)=\phi(a^n)=a^{tn}$. We have $t=1$ and $\phi(a)=a,\phi(b)= b^p a^q$. $\fix \phi$ is a cyclic subgroup generated by $a$.

Subcase (2.2). If $p=1$, then $\phi(a)=a^t,\phi(b)= b a^q$.

If $t=1$ and $q=0$, then $\phi=id$ and $\fix \phi=G$.

If $t=1$ and  $q\neq 0$, then $\fix \phi$ is generated by $a$ and $b^2$ which is isomorphic to a rank two free abelian group $\Z\oplus \Z$.

If $t \neq 1$, then we have
 $$b^m a^n=g=\phi(g)=\phi(b^m a^n)=(b a^q)^m a^{tn}=\left\{\begin{array}{ll}
                        b^m a^{q+tn},  &  m\mbox{ is odd} \\
                        b^m a^{tn},  &  m \mbox{ is even}
                      \end{array}\right..$$
Note that for any $k\in \Z$, $b^2=(ba^k)(ba^k)$. So if ${\frac{q}{1-t}}\in \Z$,  then $\fix\phi$ is a cyclic subgroup generated by $b a^{\frac{q}{1-t}}$; if ${\frac{q}{1-t}}$ is not an integer then $\fix\phi$ is a cyclic subgroup generated by $b^2$.

In conclusion, we have prove that $\fix \phi$ is either $G$, $\langle a,b^2\rangle\cong\Z\oplus \Z$, $\Z$ or trivial for any endomorphism $\phi$ of $G$. So $\fix\mathcal B$ is also one of such subgroups for any family of endomorphisms $\mathcal B$.
\end{exam}

The following example shows that Theorem \ref{main thm} is sharp.

\begin{exam}
Let the surface group $G=\langle a_1,b_1,\ldots, a_g,b_g|\prod_{i=1}^g [a_i,b_i] \rangle$. Consider the automorphism $\phi_n: G\to G$ induced by a Dehn twist:
$$a_i\mapsto a_i,\quad i=1,\ldots ,g;$$
$$b_j\mapsto b_j,\quad j=1,\ldots,g-1;\quad\quad b_g\mapsto a_g^nb_g.$$
Then
$$\bigcap_{n=1}^\infty\fix\phi_n=\langle a_1,b_1,\ldots, a_{g-1},b_{g-1}, a_g\rangle\cong F_{2g-1},$$
a free group with rank $2g-1$.
\end{exam}

The example below shows that the intersection of two retracts of a surface group is not a retract, which is not similar to the case of free groups, see Proposition \ref{intersection of retracts of free gp}.

\begin{exam}
Let $G=\langle a,b,c,d|a^{-1}b^{-1}abc^{-1}d^{-1}cd\rangle$ be a surface group, and
$$A=\langle a,b\rangle<G, \quad\quad B=\langle c,d\rangle<G.$$
Note that
$$\phi: G\to A,\quad a,d\mapsto a, \quad b,c\mapsto b$$
and
$$\psi: G\to B,\quad a,d\mapsto d, \quad b,c\mapsto c$$
are two retractions. Then $A$ and $B$ are two retracts of $G$. But the intersection
$$A\cap B=\langle a^{-1}b^{-1}ab\rangle<G$$
is not a retract. Indeed, if there is a retraction $\pi: G\to \langle a^{-1}b^{-1}ab\rangle$, then
$$\pi(a^{-1}b^{-1}ab)=\pi(a^{-1})\pi(b^{-1})\pi(a)\pi(b)=\pi(a^{-1})\pi(a)\pi(b^{-1})\pi(b)=1\in \langle a^{-1}b^{-1}ab\rangle,$$
the second equality holds since $\langle a^{-1}b^{-1}ab\rangle$ is free. It contradicts to $\pi|_{\langle a^{-1}b^{-1}ab\rangle}=id$.
\end{exam}

On retracts of surface groups, we have a question below generalized from \cite[Question 20]{B}: Is every retract $R$ of a finitely generated free group $F$ inert in $F$?

\begin{ques}\label{inert question}
Is every retract $H$ of a surface group $G$ inert in $G$? Namely, is
$$\rk (H\cap K)\leq \rk K$$
for any subgroup $K\leq G$?
\end{ques}

If $K$ is also a surface group, then the answer is affirmative.

Indeed, we have a finite covering $p: \tilde S\to S$ of closed surfaces such that $G=\pi_1(S)$ and $K=p_*(\pi_1(\tilde S))\leq G$. Since $H$ is a proper retract of $G$, $H$ is a free subgroup with $\rk H\leq \frac{1}{2} \rk G$ according to Lemma \ref{intersection of two retracts}. Thus there is a noncompact surface $F'$ and a covering $f': F'\to S$ such that $H=f'_*(\pi_1(F'))$. Pick a compact incompressible subsurface $F\subset F'$ such that $\pi_1(F)=\pi_1(F')$, then the map $f=f'|_F: F\to S$ is $\pi_1$-injective and $H=f_*(\pi_1(F))$. As in the proof of Theorem \ref{inert of geometric subgp}, consider the pull back map $p': M\to F$ of $p$ and $f$, where
$$M=\{(x,y)\in F\times \tilde S|f(x)=p(y)\}$$
such that $p'((x,y))=x$. Since $p: \tilde S\to S$ is a finite covering, $p'$ is also a finite covering.
Let $M_0\subseteq M$ be the component containing the base point. Then $p'|_{M_0}:M_0\to F$ is a covering of compact surfaces with nonempty boundary of sheets $\chi(M_0)/\chi(F)\leq \chi(\tilde S)/\chi(S)$. It implies that
$$\frac{1-\rk\pi_1(M_0)}{1-\rk\pi_1(F)}\leq \frac{2-\rk\pi_1(\tilde S)}{2-\rk\pi_1(S)}.$$
Note that $H\cap K=f_*p'_*(\pi_1(M_0))\cong \pi_1(M_0),$
thus
$$\rk(H\cap K)-1\leq \frac{(\rk H-1)(\rk K-2)}{\rk G-2}.$$
Hence
$$\rk(H\cap K)\leq \frac{1}{2}\rk K\leq \rk K.$$

If $K$ is free, and the subgroup $\langle H,K\rangle\leq G$ generated by $H$ and $K$ is also free, then Question \ref{inert question} becomes \cite[Question 20]{B}).

If $K$ is free, and $\langle H,K\rangle\leq G$ is a surface group, what will happen?

\vspace{6pt}

\noindent\textbf{Acknowledgements.} The authors would like to thank Hongbin Sun for valuable communications. This work was carried out while the second author was visiting Princeton University and he would like to thank for their hospitality. The first author is supported by NSFC (No. 11001190 and No. 11271276), and the second author is partially supported by NSFC (No. 11201364) and ``the Fundamental Research Funds for the Central Universities".

%References--------------------------------------------------------------------------------------------------------------------------

\end{document}